\documentclass[12pt,reqno]{amsart}
\usepackage{amsmath,amsfonts,amsthm,amsopn,amssymb}
\usepackage{cite,marginnote}
\usepackage{color,enumitem,graphicx}
\usepackage[colorlinks=true,urlcolor=blue,
citecolor=red,linkcolor=blue,linktocpage,pdfpagelabels,
bookmarksnumbered,bookmarksopen]{hyperref}
\usepackage[english]{babel}

\usepackage[left=2.9cm,right=2.9cm,top=2.8cm,bottom=2.8cm]{geometry}
\usepackage[hyperpageref]{backref}

\numberwithin{equation}{section}

\makeindex
\newtheorem{theorem}{Theorem}[section]
\newtheorem{lemma}{Lemma}[section]
\newtheorem{corollary}{Corollary}[section]
\newtheorem{proposition}{Proposition}[section]
\newtheorem{remark}{Remark}[section]
\newtheorem{example}{Example}[section]
\newtheorem{definition}{Definition}[section]

\newcommand{\ud}{\mathrm{d}}
\newcommand{\RN}{\mathbb R^N}

\newcommand{\iy}{\infty}

\newcommand{\s}{\section}
\newcommand{\dd}{\delta}
\newcommand{\DD}{\Delta}
\newcommand{\g}{\gamma}
\newcommand{\G}{\Gamma}
\newcommand{\na}{\nabla}

\newcommand{\pa}{\partial}

\newcommand{\R}{\mathbb R}
\newcommand{\al}{\alpha}

\newcommand{\ti}{\tilde}

\newcommand{\re}[1]{\eqref{#1}}

\newcommand{\rg}{\rightarrow}

\newcommand{\e}{\varepsilon}
\newcommand{\vp}{\varphi}

\newcommand{\lab}{\label}
\newcommand{\bt}{\begin{theorem}}
\newcommand{\et}{\end{theorem}}
\newcommand{\bl}{\begin{lemma}}
\newcommand{\el}{\end{lemma}}
\newcommand{\bd}{\begin{definition}}
\newcommand{\ed}{\end{definition}}
\newcommand{\bc}{\begin{corollary}}
\newcommand{\ec}{\end{corollary}}
\newcommand{\bp}{\begin{proof}}
\newcommand{\ep}{\end{proof}}
\newcommand{\bx}{\begin{example}}
\newcommand{\ex}{\end{example}}
\newcommand{\bi}{\begin{exercise}}
\newcommand{\ei}{\end{exercise}}
\newcommand{\bo}{\begin{proposition}}
\newcommand{\eo}{\end{proposition}}
\newcommand{\br}{\begin{remark}}
\newcommand{\er}{\end{remark}}
\newcommand{\be}{\begin{equation}}
\newcommand{\ee}{\end{equation}}
\newcommand{\ba}{\begin{align}}
\newcommand{\ea}{\end{align}}
\newcommand{\bn}{\begin{enumerate}}
\newcommand{\en}{\end{enumerate}}
\newcommand{\bg}{\begin{align*}}
\newcommand{\bcs}{\begin{cases}}
\newcommand{\ecs}{\end{cases}}

\newcommand{\vr}{\varepsilon}

\newcommand{\bean}{\begin{eqnarray*}}
\newcommand{\eean}{\end{eqnarray*}}

\title[Multi-peak solutions]{Multi-peak solutions for nonlinear Choquard equation with a general nonlinearity}

\author[M. B.\ Yang]{Minbo Yang}
\author[J. J.\ Zhang]{Jianjun Zhang}
\author[Y. M.\ Zhang]{Yimin Zhang}

\address[M. B.\ Yang]{\newline\indent Department of Mathematics
\newline\indent
Zhejiang Normal University
\newline\indent
Jinhua 321004, PR China}
\email{\href{mailto:mbyang@zjnu.edu.cn}{mbyang@zjnu.edu.cn}}

\address[J. J.\ Zhang]{\newline\indent College of Mathematics and Statistics
\newline\indent
Chongqing Jiaotong University
\newline\indent
Chongqing 400074, PR China}
\email{\href{mailto:zhangjianjun09@tsinghua.org.cn}{zhangjianjun09@tsinghua.org.cn}}

\address[Y. M.\ Zhang]{\newline\indent Wuhan Institute of Physics and Mathematics
\newline\indent
Chinese Academy of Sciences
\newline\indent
Wuhan 430071, China}
\email{\href{mailto:zhangym802$@$126.com}{zhangym802$@$126.com}}

\thanks{J.J. Zhang is the corresponding author and was partially supported by the Science Foundation of Chongqing Jiaotong University(15JDKJC-B033). M.B. Yang was partially supported by NSFC (11101374, 11271331, 11571317) and ZJNSF(LY15A010010). Y.M. Zhang was partially supported by NSFC (11271360, 11471330).}


\begin{document}

\begin{abstract}
In this paper, we study a class of nonlinear Choquard type equations involving a general nonlinearity. By using the method of penalization argument, we show that there exists a family of solutions having multiple concentration regions which concentrate at the minimum points of the potential $V$. Moreover,
the {\it monotonicity} of $f(s)/s$ and the so-called {\it Ambrosetti-Rabinowitz} condition are not required.
\end{abstract}

\maketitle

\s{Introduction}
\renewcommand{\theequation}{1.\arabic{equation}}

\noindent This paper is concerned with the following nonlinear Choquard equation
\begin{equation}\label{q1}
-\e^2\DD v+V(x)v=\e^{-\alpha}(I_\al\ast F(v))f(v),\,\, v\in H^1(\RN),
\end{equation}
where $N\geq 3$, $\alpha\in(0,N)$, $F$ is the prime function of $f$ and $I_\al$ is the Riesz potential defined for every $x\in\RN\setminus\{0\}$ by
$$
I_\al(x):=\frac{\G((N-\al)/2)}{\G(\al/2)\pi^{N/2}2^\al|x|^{N-\al}}.
$$
In the sequel, we assume that
the potential function $V$ satisfies the following conditions:
\begin{itemize}

\item [(V1)] $V\in C(\RN,\R)$ and $\inf_{x\in\RN}V(x)=1$.

\item [(V2)] There are bounded disjoint open sets $O^i, i=1,2,\cdots,k$, such that for any $i\in\{1,2,\cdots,k\}$, $$0<m_i\equiv\inf_{x\in O^i}V(x)<\min_{x\in \partial O^i}V(x),$$

\end{itemize}
and $f\in C(\R^+,\R^+)$ satisfies
\begin{itemize}
\item [(F1)] $\lim_{t\rg 0^+}f(t)/t=0$;
\item [(F2)] $\lim_{t\rg +\infty}f(t)/t^{\frac{\al+2}{N-2}}=0$;
\item [(F3)] there exists $s_0>0$ such that $F(s_0)>0$.
\end{itemize}
For any $i\in\{1,2,\cdots,k\}$, let $$\mathcal{M}^i\equiv\{x\in O^i: V(x)=m_i\}.$$

\noindent The main theorem of this paper reads as
\bt\lab{Theorem 1} Suppose that $\al\in((N-4)_+,N)$, $(V1)$-$(V2)$ and $(F1)$-$(F3)$.
Then, for sufficiently small $\e>0$, \re{q1} admits a positive solution $v_{\e}$, which satisfies
\begin{itemize}
\item [(i)] there exist $k$ local maximum points $x_\e^i\in O^i$of $v_\e$ such that $$\lim_{\e\rg 0}\max_{1\le i\le k}dist(x_\e^i,\mathcal{M}^i)=0,$$ and $w_\e(x)\equiv v_\e(\e x+x_\e^i)$ converges (up to a subsequence) uniformly to a least energy solution of
\be\lab{q2} -\DD u+m_iu=(I_\al\ast F(u))f(u),\ \ \ \ u\in H^1(\RN); \ee
\item [(ii)] $v_\e(x)\le C\exp(-\frac{c}{\e}\min\limits_{1\le i\le k}|x-x_\e^i|)$ for some $c,C>0$.
\end{itemize}
\et

Our motivation for the study of such a problem goes back at least to the pioneering work of Floer and Weinstein \cite{F-W} (see also \cite{Oh}) concerning the Schr\"odinger equation
\be\lab{sch}
-\e^2\DD u+V(x)u=f(u),\,\,\ x\in\RN.
\ee
By means of a Lyapunov-Schmidt reduction approach, these authors constructed sing-peak or multi-peak solutions of \re{sch} concentrating around any given non-degenerate critical points of $V$ as $\e\rg0$. For $\e>0$ sufficiently small, these standing waves are
referred to as semi-classical states, which describe the transition from quantum mechanics to classical mechanics. For the
detailed physical background, we refer to \cite{Oh} and references therein. In \cite{F-W,Oh}, their arguments are
based on a Lyapunov-Schmidt reduction in which a non-degenerate condition plays a crucial role. Without such a non-degenerate condition, by using the mountain pass argument, Rabinowitz \cite{Rab} proved the existence of positive solutions of \re{sch} for small $\e>0$ provided the following global potential well condition
$$
\liminf_{|x|\rg\iy}V(x)>\inf_{\RN}V(x)
$$
holds. Subsequently, by virtue of a penalization approach, del Pino and Felmer
\cite{DF} established the existence of a single-peak solution to \re{sch} which
concentrates around local minimum points of $V$. Some related results can be found in \cite{WX,Ni-Wei,Felmer2,Pino,DPR,Alves1} and the references therein.
In the works above, the nonlinearity $f$ satisfies the monotonicity condition
\be\lab{N}\tag{N}
\hbox{$f(s)/|s|$ is strictly increasing for $s\not=0$}
\ee
or the well-known $Ambrosetti$-$Rabinowitz\ condition$
\be\lab{AR}\tag{AR}
\hbox{$0<\mu\int_0^sf(t)\,\ud t\le sf(s)$ for any $t\not=0$ and some $\mu>2$.}
\ee
To attack the existence of positive solutions to \re{sch} without \re{N} and \re{AR}, by introducing a new penalization approach, Byeon and Jeanjean \cite{byeon} constructed a spike solution near local minimal points of $V$ under a almost optimal hypotheses:
\begin{itemize}

\item [$({BL1})$] $f\in C(\R, \R)$ and $\lim\limits_{t\to 0}f(t)/t=0;$

\item [$({BL2})$]  there exists $p\in (1, (N+2)/(N-2))$ such that $\limsup\limits_{t\to \infty}f(t)/t^p<\infty;$

\item [$({BL3})$]  there exists $T>0$ such that $mT^2/2<F(T) :\equiv\int_0^Tf(t)dt.$

\end{itemize}
$(BL1)$-$(BL3)$ are referred to as the Berestycki-Lions\ conditions, which were
firstly proposed by a celebrated paper \cite{Lions}. We refer the reader to \cite{BT1,BT2,byeon4} and the references therein for the development on this subject.  \vskip0.1in

Taking $u(x)=v(\e x)$ and $V_{\e}(x)=V(\e x)$, then \eqref{q1} is
equivalent to the following problem
\begin{eqnarray}\label{q2}
-\DD u+V_{\e}(x)u=(I_\al\ast F(u))f(u),\ \ u\in H^1(\RN).
\end{eqnarray}
Obviously, the term $(I_\al\ast F(v))f(v)$ is nonlocal. Equation \re{q2} can be seen as a special case of the generalized nonlocal
Schr\"odinger equation
\begin{equation}\label{q1r}
-\DD u+V_{\e}(x)u=(K(x)\ast F(u))f(u),\ \ u\in H^1(\RN).
\end{equation}
From the view of physical background, $K(x)$ is called as a response function which possesses the information on the mutual interaction between the bosons. In general, the following equation for $a>0$ is considered as the limiting
equation of \eqref{q2}
\begin{eqnarray}\label{lb1}
-\Delta u+au=(I_\al\ast F(u))f(u),\ \ \ u\in H^1(\RN).
\end{eqnarray}
For $N=3$, $\al=2$ and $f(s)=s$, \eqref{q1} and \eqref{lb1} reduce to
\begin{equation}\label{main equation5}
-\e^2\DD v+V(x)v=\e^{-2}(I_2\ast v^2)v/2,\ \ x\in \R^3
\end{equation}
and
\begin{eqnarray}\label{limit equation2}
-\Delta u+au=(I_2\ast u^2)u/2, \ x\in \R^3.
\end{eqnarray}
Equation \eqref{limit equation2} is commonly named as the stationary
Choquard equation. In 1976, during the symposium on
Coulomb systems at Lausanne, Choquard proposed this type of equations as an approximation to Hartree-Fock
theory for a one component plasma\cite{LS}. It arises in multiple
particle systems\cite{Gross, LS}, quantum
mechanics\cite{Penrose1,Penrose2,Penrose3} and laser beams,
etc. In the recent years, There has been a considerable attention to be paid on investigating the Choquard equation. In the pioneering works \cite{Lieb1}, Lieb investigated the existence and uniqueness of positive solutions to equation \eqref{limit equation2}. Subsequently, Lions\cite{Lions2,Lions3} obtained the existence and multiplicity results for \eqref{limit equation2} via the critical point theory. In \cite{MZ}, Ma and Zhao studied the classification of all positive solutions to the nonlinear Choquard problem
\begin{eqnarray}\label{limit equation3}
-\Delta u+u=(|x|^{-\alpha}\ast |u|^{p}) |u|^{p-2}u, \ u\in H^1(\R^{N}),
\end{eqnarray}
where $\al\in(0,N)$ and $p\in[2,(2N-\al)/(N-2))$. Due to the present of the nonlocal term, the standard
method of moving planes cannot be used directly. So the classification of positive solutions to \re{limit equation3}(even for $p=2$)
had remained as an longstanding open problem. By using the integral form of the
method of moving planes introduced by Chen et al. \cite{Chen}, Ma and Zhao\cite{MZ} solved this open problem. Precisely, they proved that up to translations,
positive solutions of equation \eqref{limit equation3} are radially symmetric and monotone decreasing, under some assumption on $\alpha$, $p$ and $N$. In \cite{MV3}, Moroz and van Schaftingen eliminated this restriction and established an optimal range of parameters for the existence of a positive
ground state solution of \eqref{limit equation3}. Moreover, they proved that all positive ground state solutoions of \eqref{limit equation3} are
radially symmetric and monotone decaying about some point. Later, in the spirit of Berestycki and Lions, Moroz and van Schaftingen\cite{MV1} gave a almost necessary conditions on the nonlinearity $f$ for the existence of ground state solutions of \eqref{q2}. The symmetry of slutions was considered in \cite{MV1} as well.
\vskip0.1in
In the present paper, we are interested in semiclassical state solutions of \re{q1}. For the special case \eqref{main equation5}, there have been many results
on this subject( see \cite{CSS, MN, Nolasco, Secchi, WW} and the references therein). By using a Lyapunov-Schmidt
reduction argument, Wei and Winter\cite{WW} proved the existence
of multibump solutions of \re{main equation5} concentrating at local minima, local maxima
or non-degenerate critical points of $V$ provided $\inf V>0$. Subsequently, Secchi\cite{Secchi} studied the case of the
potential $V>0$ and satisfying $\liminf_{|x|\rg\iy}V(x)|x|^\g>0$ for some $\g\in[0,1)$. By a perturbation technique, they obtained the existence of
positive bound state solution concentrating at local minimum (or maximum) points of V when $\e\rightarrow 0$.
Moroz and Van Schaftingen \cite{MV2} considered the semiclassical states of the Choquard equation \re{q1} with $f(s)=|s|^{p-2}s$,
$p\in[2,(N+\al)/(N-2)_+)$. By introducing a novel nonlocal penalization technique, the authors proved that \re{q1} has a family of solutions
concentrating at the local minimum of $V$. Moreover, in \cite{MV2} the potential $V$ maybe vanishes at infinity, and the assumptions on the decay of $V$
and the admissible range for $p\ge2$ are optimal. In \cite{YD}, Yang and Ding considered the following equation
 \begin{equation}\label{1.8}
 -\varepsilon^2\Delta u +V(x)u  =\Big[\frac{1}{|x|^{\mu}}\ast u^p\Big]u^{p-1}, \quad \mbox{in} \quad \R^3.
\end{equation}
By using the variational methods, for suitable parameters $p, \mu$, the authors  obtained the existence of solutions of (\ref{1.8}).
By the penalization method in \cite{DF}, Alves and Yang \cite{alves1} considered the concentration behavior of solutions to the following generalized quasilinear Choquard equation
$$
-\e^p\DD_p v+V(x)|v|^{p-2}v=\e^{\mu-N}\left(\int_{\RN}\frac{Q(y)F(v(y))}{|x-y|^\mu}\,\ud y\right)Q(x)f(v),\,\, x\in\RN,
$$
where $\DD_p$ is the $p$-Laplacian operator, $p\in(1,N)$ and $\mu\in(0,N)$. For the related results, we refer to \cite{alves,CSS,Sun} and the references therein.
\vskip0.1in
To sum  up, in all the works mentioned above, the authors only considered the Choquard equation \re{q1} with a power type nonlinearity or a general nonlinearity
satisfying
some sort of monotonicity condition or Ambrosetti-Rabinowtiz type condition. Similar to \cite{byeon} for the local problem \re{sch}, it seems natural to ask

{\it Does the similar concentration phenomenon occur for the Choquard equation \re{q1} under very mild assumptions on $f$ in the spirit of Berestycki and Lions?}

\noindent In the present paper, we will give an affirmative answer to this question. In particular, the monotonicity condition and Ambrosetti-Rabinowtiz condition are not
required.
\vskip0.1in
The spirit of this paper is somewhat akin to \cite{bj,byeon}. The penalization argument is used to prove Theorem \ref{Theorem 1}.
This method is widely used by many authors. The penalization functional we need was first introduced by Byeon and Wang in \cite{BW}.


\s{Proof of Theorem \ref{Theorem 1}}

\renewcommand{\theequation}{2.\arabic{equation}}

\noindent  In this section, we will use the framework of Byeon and Jeanjean \cite{bj}(see also \cite{byeon}) to prove our main result.
\subsection{The limit problem}
We define an energy functional for the limiting problem \re{lb1} by
$$
L_a(u)=\frac{1}{2}\int_{\RN}|\na u|^2+au^2-\frac{1}{2}\int_{\RN} (I_\al\ast F(u))F(u),\ \ u\in H^1(\RN).
$$
Let $E_a$ be the least energy of \re{lb1} and $S_a$ be the set of least energy solutions $U$ of \re{lb1} satisfying $U(0)=\max_{x\in \RN}U(x)$, the following property of  $S_a$ was proved in \cite{MV1}.

\bo\lab{prop2.1}{\rm\cite{MV1}}\ Assume that $f$ satisfies $(F1)$-$(F3)$, then
\begin{itemize}
\item [$(i)$] $S_a\not=\emptyset$ and $S_a$ is compact in $H^1(\RN)$.
\item [$(ii)$] $E_a$ coincides with the mountain pass value.
\item [$(iii)$] For any $U\in S_a$, $U\in W_{loc}^{2,q}(\RN)$ for any $q\ge1$. Moreover,
 \begin{itemize}
 \item [$(iv)$] $U$ is radially symmetric and radially decreasing.
 \item [$(v)$] $U$ satisfies the Pohoz\v{a}ev identity:
 $$
 \frac{N-2}{2}\int_{\mathbb{R}^N}|\na U|^2+\frac{N}{2}a\int_{\mathbb{R}^N}U^2=\frac{N+\al}{2}\int_{\mathbb{R}^N}(I_\al\ast F(U))F(U).
 $$
 \end{itemize}
\end{itemize}
\eo
\noindent Now, we give some further estimates about the boundedness and decay for any $U\in S_a$. The following Hardy-Littlewood-Sobolev inequality will be used frequently later.
\bl\lab{hls}{\rm \cite[Theorem 4.3]{LL}}
Let $s, r>1$ and $0<\al<N$ with $1/s+1/r=1+\al/N$, $f\in L^s(\RN)$ and $h\in L^r(\RN)$. There exists a constant $C(s, N, \al, r)$ (independent of $f, h$) such that
$$
\left|\int_{\RN}\int_{\RN}f(x)|x-y|^{N-\al}g(y)\,\ud x\ud y\right|\le C(s, N, \al)\|f\|_s\|g\|_r,
$$
where the sharp constant $C(s, N, \al)$ satisfies
$$
C(s, N, \al)\le \frac{N}{sr\al}(|\mathbb{S}^{N-1}|/N)^{1-\al/N}\left[(\frac{1-\al/N}{1-1/s})^{1-\al/N}+(\frac{1-\al/N}{1-1/r})^{1-\al/N}\right].
$$
\el
Now, we adopt some ideas from \cite{MV1,alves} to give the decay of the ground state solutions to \re{lb1}.
\bo\lab{decay}
$S_a$ is uniformly bounded in $L^\iy(\RN)$. Moreover, there exist $C,c>0$, independent of $u\in S_a$, such that $|D^{\al_1}u(x)|\le C\exp(-c|x|), \,x\in \RN$ for $|\al_1|=0,1$.
\eo
\bp First, we give the uniformly boundedness of $u\in S_a$. For any $u\in S_a$, we get
$$
L_a(u)=\frac{2+\al}{2(N+\al)}\int_{\RN}|\na u|^2+\frac{\al a}{2(N+\al)}\int_{\RN} u^2=E_a,
$$
which implies that $S_a$ is bounded in $H^1(\RN)$. Let $H(u)=F(u)/u$ and $K(u)=f(u)$, then by $(F1)$-$(F2)$, there exists $C>0$(independent of $u$) such that
$$
|H(u(x))|,|K(u(x))|\le C\left(|u(x)|^{\al/N}+|u(x)|^{(\al+2)/(N-2)}\right),\,\,x\in\RN.
$$
It follows that $H(u),K(u)\in L^{2N/\al}(\RN)+L^{2N/(\al+2)}(\RN)$, i. e, $H(u)=H^\ast(u)+H_\ast(u)$, $K(u)=K^\ast(u)+K_\ast(u)$ with $H^\ast(u), K^\ast(u)\in L^{2N/\al}(\RN)$ and $H_\ast(u), K_\ast(u)\in L^{2N/(\al+2)}(\RN)$. Moreover, $H^\ast(u), K^\ast(u)$ are uniformly bounded in $L^{2N/\al}(\RN)$ for any $u\in S_a$. So is $H_\ast(u), K_\ast(u)$ in $L^{2N/(\al+2)}(\RN)$. Then by \cite[Proposition 3.1]{MV1}, for any $u\in S_a$ we get $u\in L^p(\RN)$ for $p\in[2,\frac{N}{\al}\frac{2N}{N-2})$. Meanwhile, there exists $C_p$(depending on only $p$) such that for any $p\in[2,\frac{N}{\al}\frac{2N}{N-2})$,
\be\lab{bounded}
\|u\|_p\le C_p\|u\|_2,\,\,\mbox{for all}\,\, u\in S_a.
\ee
Now, we claim that $I_\al\ast F(u)$ is uniformly bounded in $L^\iy(\RN)$ for $u\in S_a$. By $(F1)$-$(F2)$, there exists $c>0$ such that $|F(\tau)|\le C(|\tau|^2+|\tau|^{(N+\al)/(N-2)})$ for all $\tau\in\R$. Then for any $x\in\RN$ and $u\in S_a$, there exists $C(\al)$ (depending only $N,\al$) such that
{\allowdisplaybreaks
\begin{align*}
(I_\al\ast |F(u)|)(x)\le& C(\al)\int_{|x-y|\ge1}\frac{|u|^2+|u|^{(N+\al)/(N-2)}}{|x-y|^{N-\al}}\,\ud y\\
\, &+C(\al)\int_{|x-y|\le1}\frac{|u|^2+|u|^{(N+\al)/(N-2)}}{|x-y|^{N-\al}}\,\ud y\\
\le& C(\al)\int_{\R^2}(|u|^2+|u|^{(N+\al)/(N-2)})\,\ud y\\
\, &+C(\al)\int_{|x-y|\le1}\frac{|u|^2+|u|^{(N+\al)/(N-2)}}{|x-y|^{N-\al}}\,\ud y.
\end{align*}}%
By $\al\ge N-4$, $(N+\al)/(N-2)\in[2,2N/(N-2))$. By \re{bounded}, there exists $c$ (independent of $u$) such that for any $x\in\RN$,
$$
(I_\al\ast |F(u)|)(x)\le C+C(\al)\int_{|x-y|\le1}\frac{|u|^2+|u|^{(N+\al)/(N-2)}}{|x-y|^{N-\al}}\,\ud y.
$$
In the following, we estimate the term
$$
\int_{|x-y|\le1}\frac{|u|^2+|u|^{(N+\al)/(N-2)}}{|x-y|^{N-\al}}\,\ud y.
$$
Choosing $t\in(\frac{N}{\al},\frac{N}{\al}\frac{N}{N-2})$ with $2t\in(2,\frac{N}{\al}\frac{2N}{N-2})$ and $(\al-N)(t/(t-1))+N>0$,
\begin{align*}
&\int_{|x-y|\le1}|x-y|^{\al-N}u^2\,\ud y\\
&\le\|u\|_{2t}^2\left(\int_{|x-y|\le1}|x-y|^{(\al-N)(t/(t-1))}\,\ud y\right)^{1-1/t}\le C_1\|u\|_{2t}^2.
\end{align*}
Choosing $s\in(\frac{N}{\al},\frac{N}{\al}\frac{2N}{N+\al})$ with $s\frac{N+\al}{N-2}\in(2,\frac{N}{\al}\frac{2N}{N-2})$ and $(\al-N)(s/(s-1))+N>0$,
\begin{align*}
&\int_{|x-y|\le1}|x-y|^{\al-N}|u|^{(N+\al)/(N-2)}\,\ud y\\
&\le\|u\|_{s\frac{N+\al}{N-2}}^{(N+\al)/(N-2)}\left(\int_{|x-y|\le1}|x-y|^{(\al-N)(s/(s-1))}\,\ud y\right)^{1-1/s}\\
&\le C_2\|u\|_{s\frac{N+\al}{N-2}}^{(N+\al)/(N-2)}.
\end{align*}
Thus by \re{bounded} $I_\al\ast |F(u)|$ is uniformly bounded in $L^\iy(\RN)$ for all $u\in S_a$. By the standard Moser iteration, $S_a$ is uniformly bounded in $L^\iy(\RN)$. Moreover, by the radial lemma, one knows $u(x)\rg0$ uniformly as $|x|\rg\iy$ for $u\in S_a$. By virtue of the comparison principle, there exist $C,c>0$, independent of $u\in S_a$, such that $|D^{\al_1}u(x)|\le C\exp(-c|x|), \,x\in \RN$ for $|\al_1|=0,1$.

\ep

\subsection{The penalization argument}

To study \re{q1}, it suffices to investigate \re{q2}. Let $H_\e$ be the completion of $C_0^\iy(\RN)$ with respect to the norm
$$
\|u\|_\e=\left(\int_{\RN} (|\na u|^2+V_\e u^2)\right)^{\frac{1}{2}}.
$$
Since $\inf_{\RN}V(x)=1$, $H_\e\subset H^1(\RN)$. For any set $B\subset \RN$ and $\e>0$, we define $B_\e\equiv\{x\in\RN: \e x\in B\}$ and $B^{\dd}\equiv\{x\in\RN: \mbox{dist}(x,B)\le\dd\}$. Let $\mathcal{M}=\cup_{i=1}^k\mathcal{M}^i$ and $O=\cup_{i=1}^k O^i$. Since we are interested in the positive solutions of \re{q1}, from now on, we may assume that $f(t)=0$ for $t \le0.$  \noindent For $u\in H_\e$, let
$$
P_\e(u)=\frac{1}{2}\int_{\RN} |\na u|^2+V_\e u^2-\frac{1}{2}\int_{\RN} (I_\al\ast F(u))F(u).
$$
Fixing an arbitrary $\mu>0$, we define
\be
\chi_\e(x)=
\bcs\nonumber
0,\ \ \ \ \mbox{if}\ \ x\in O_\e,\\
\varepsilon^{-\mu},\ \ \mbox{if}\ \ x\in \RN\setminus O_\e,
\ecs
\chi_\e^i(x)=
\bcs\nonumber
0,\ \ \ \ \mbox{if}\ \ x\in O_\e^i,\\
\varepsilon^{-\mu},\ \ \mbox{if}\ \ x\in \RN\setminus O_\e^i,
\ecs
\ee
and
$$
Q_\e(u)=\left(\int_{\RN}\chi_\e u^2\, \ud x-1\right)_+^2,\ \ Q_\e^i(u)=\left(\int_{\RN}\chi_\e^i u^2\, \ud x-1\right)_+^2.
$$
Let $\G_\e,\G_\e^i(i=1,2,\cdots,k):H_\e\rg \R$ be
given by $$\G_\e(u)=P_\e(u)+Q_\e(u),\G_\e^i(u)=P_\e(u)+Q_\e^i(u).$$ It is standard to check that $\G_\e,\G_\e^i\in
C^1(H_\e)$. To find solutions of \re{q2} which concentrate in $O$ as $\e\rg 0$, we shall search critical points of
$\G_\e$ such that $Q_\e$ is zero. The functional $Q_{\e}$ that was first introduced in \cite{BW}, will act as a penalization to force the concentration phenomena to occur inside $O$.

\noindent Now, we construct a set of approximate solutions of \re{q2}. Let
$$
\delta=\frac{1}{10}\min\{\mbox{dist}(\mathcal{M},O^c),\min\limits_{i\not=j}\mbox{dist}(O^i,O^j)\}.
$$ We fix a $\beta\in (0,\delta)$ and a cut-off $\vp\in C_0^\infty(\RN)$ such that $0\le\vp\le1,\vp(x)=1$ for $|x|\le \beta$ and $\vp(x)=0$ for $|x|\ge 2\beta$. Let $\vp_{\e}(y)=\vp(\e y), y\in\RN$ and for some $x_i\in (\mathcal{M}^i)^{\beta}$, $1\le i\le k$, and $U_i\in S_{m_i}$, we define
$$
U_{\e}^{x_1,x_2,\cdots,x_k}(y)=\sum_{i=1}^k\vp_{\e}\left(y-\frac{x_i}{\e}\right)U_i\left(y-\frac{x_i}{\e}\right).
$$
As in \cite{bj}, we will find a solution near the set
$$
X_{\e}=\{U_{\e}^{x_1,x_2,\cdots,x_k}\,\,|\,\,x_i\in (\mathcal{M}^i)^{\beta}, U_i\in S_{m_i},i=1,2,\cdots,k\}
$$
for sufficiently small $\e>0$. For each $1\le i\le k$, Choosing some $U_i\in S_{m_i}$ and $x_i\in\mathcal{M}^i$ but fixed, define
$$
W_{\e,t}^i(\cdot)\equiv(\vp_\e U_{i,t})(\cdot-x_i/\e)\, \, \mbox{where}\, \, U_{i,t}(\cdot)=U_i(\cdot/t).
$$
\bl\lab{nt}
There exist $T_i>0, i=1,2,\cdots,k$, such that for $\e>0$ small enough, $\Gamma_{\e}(W_{\e,T_i}^i)<-2, i=1,2,\cdot,k$.
\el
\bp
By a direct calculation, we get  $\G_\e(W_{\e,t}^i)=P_\e(W_{\e,t}^i)$ for any $t>0$ and for any $1\le i\le k$,
$$
L_{m_i}(U_{i,t})=\frac{t^{N-2}}{2}\int_{\RN}|\nabla U_i|^2+\frac{t^{N}}{2}m_i\int_{\RN}|U_i|^2-\frac{t^{N+\al}}{2}\int_{\RN}(I_\al\ast F(U_i))F(U_i).
$$
Then there exists $T_i>1$ such that $L_{m_i}(U_{i,T_i})<-2$ for $t>T_i$. Notice that
$$
P_\e(W_{\e,t}^i)=L_{m_i}(W_{\e,t}^i)+\frac12\int_{\RN}(V_\vr(x)-m_i)(W_{\e,t}^i)^2,
$$
we have
$$
P_\e(W_{\e,t}^i)=L_{m_i}(U_{i,t})+O(\e),
$$
by the decay property of $U_{i}$ in proposition \ref{decay}. Consequently, we know that $\G_\e(W_{\e,T_i}^i)<-2$ for $\e>0$ small.
\ep
For any $1\le i\le k$, let $\gamma_\e^i(t)(\cdot)=W_{\e,t}^i(\cdot)$ for $t>0$. Due to $N\ge3$, $\lim\limits_{t\rightarrow0}\|W_{\e,t}^i\|_{\e}=0$, let
$\gamma_{\e}^i(0)=0$. Then as in \cite{bj}, we define a min-max value $C_\e^i:$
$$C_\e^i=\inf_{\vp\in \Phi_\e^i}\max_{s_i\in [0,T_i]}\G_\e^i(\vp(s_i)),$$
where $\Phi_\e^i=\{\vp\in C([0,T_i],H_{\e}):
\vp(0)=0,\vp(T_i)=\gamma_\e^i(T_i)\}$.
Similar to Proposition 2 and 3 in \cite{byeon}, we have
\bo\lab{prop2.2} For any $1\le i\le k$, we have
$$\lim\limits_{\e\rg 0}C_\e^i=E_{m_i}.$$
\eo
\noindent Finally, let
$$
\gamma_\e(s)=\sum_{i=1}^k\gamma_\e^i(s_i),\, s=(s_1,s_2,\cdots,s_k)
$$
and $$D_\e=\max_{s\in T}\G_\e(\gamma_\e(s)),$$ where $T\equiv[0,T_1]\times\cdot\times[0,T_k]$. Since $\mbox{supp}(\gamma_\e(s))\subset \mathcal{M}^{\beta}_{\e}$ for each $s\in T$, it follows that
$$
\G_\e(\gamma_\e(s))=P_\e(\gamma_\e(s))=\sum_{i=1}^kP_\e(\gamma_\e^i(s)).
$$
By the Pohoz$\check{a}$ev's identity, for any $1\le i\le k$, we have
$$
L_{m_i}(U_{i,t})=\left(\frac{t^{N-2}}{2}-\frac{N-2}{N+\al}\frac{t^{N+\al}}{2}\right)\int_{\RN}|\nabla U_i|^2+\left(\frac{t^{N}}{2}-\frac{N}{N+\al}\frac{t^{N+\al}}{2}\right)m_i\int_{\RN}|U_i|^2.
$$
Let $$g_1(t)=\frac{t^{N-2}}{2}-\frac{N-2}{N+\al}\frac{t^{N+\al}}{2},\,\ g_2(t)=\frac{t^{N}}{2}-\frac{N}{N+\al}\frac{t^{N+\al}}{2},$$
then it is easy to know $g_j'(t)>0$ for $t\in(0,1)$ and $g_j'(t)<0$ for $t>1$, $j=1,2$. Thus, for any $1\le i\le k$, $L_{m_i}(U_{i,t})$ achieves a unique maximum point at $t=1$ for $t>0$, i. e.,
$$
\max_{t>0}L_{m_i}(U_{i,t})=L_{m_i}(U_{i})=E_{m_i},
$$
which leads to the following conclusion.
\bo\lab{prop2.3}\noindent
\begin{itemize}\item [$(i)$] $\lim_{\e\rg0}D_\e=\sum_{i=1}^kE_{m_i}:=E$;
\item [$(ii)$] $\limsup_{\e\rg0}\max_{s\in\pa T}\G_\e(\gamma_\e(s))\le\ti{E}$, where $\ti{E}=\max_{1\le j\le k}(\sum_{i\not=j}E_{m_i})$;
\item [$(iii)$] for any $d>0$, there exists $\al_0>0$ such that for $\e>0$ small,
$$\hbox{$\G_\e(\gamma_\e(s))\ge D_\e-\al_0$ implies that $\gamma_\e(s)\in X_\e^{d/2}$.}$$
\end{itemize}
\eo

\noindent Now define
$$
\G_\e^\al:=\{u\in H_\e: \G_\e(u)\le\al\}
$$
and for a set $A\subset H_\e$ and $\al>0$, let
$$
A^\al:=\{u\in H_\e: \inf_{v\in A}\|u-v\|_\e\le\al\}.
$$
In the following, we will construct a special PS-sequence of $\G_\e$, which is localized in some neighborhood $X_\e^d$ of $X_\e$.

\begin{proposition}\label{prop5}
Let $\{\e_j\}$ with $\lim\limits_{j\rightarrow
\infty}\e_j=0$, $\{u_{\e_j}\}\subset X_{\e_j}^d$
be such that
\begin{equation}\label{4sec27}
\lim\limits_{j\rightarrow
\infty}\Gamma_{\e_j}(u_{\e_j})\leq E \ \ \text{and} \ \
\lim\limits_{j\rightarrow
\infty}\Gamma_{\e_j}^{'}(u_{\e_j})=0.
\end{equation}
Then for sufficiently small $d>0$, there exist, up to a subsequence,
$\{y_j^i\}\subset \R^N$, $i=1,2,\cdots,k$, points $x^i\in \mathcal
{M}^i$, $U_i\in S_{m_i}$ such that
\begin{equation}\label{4sec28}
\lim\limits_{j\rightarrow \infty}\left|\e_jy_j^i-x^i\right|=0,
\end{equation}
and \begin{equation}\label{4sec29} \lim\limits_{j\rightarrow
\infty}\left\|u_{\e_j}-\sum\limits_{i=1}^k\varphi_{\e_j}(\cdot-y_j^i)U_i(\cdot
-y_j^i)\right\|_{\e_j}=0.
\end{equation}
\end{proposition}

\begin{proof}
Without confusion, we write $\e$ for $\e_j$. Since
$S_{m_i}$ is compact, then there exist $Z_i\in S_{m_i}$,
$x_{\e}^i\in (\mathcal {M}^i)^{\beta}$, $x^i\in (\mathcal
{M}^i)^{\beta}$, $i=1,2,\cdots,k$, $\lim\limits_{\e\rightarrow
0}x_{\e}^i=x^i$, such that up to a subsequence, denoted still by
$\{u_{\e}\}$ satisfying that for sufficiently small $\e>0$,
\begin{equation}\label{4sec30}
\left\|u_{\e}-\sum\limits_{i=1}^k\varphi_{\e}\left(\cdot
-\frac{x_{\e}^i}{\e}\right)Z_i\left(\cdot
-\frac{x_{\e}^i}{\e}\right)\right\|_{\e}\leq 2d.
\end{equation}
Set $u_{1,\e}(x)=\sum\limits_{i=1}^k\varphi_{\e}\left(x
-\frac{x_{\e}^i}{\e}\right)u_{\e}$,
$u_{2,\e}(x)=u_{\e}(x)-u_{1,\e}(x)$.

{\bf Step 1.} We claim that
\begin{equation}\label{4sec31}
\Gamma_{\e}(u_{\e})\geq
\Gamma_{\e}(u_{1,\e})+\Gamma_{\e}(u_{2,\e})+O(\e).
\end{equation}
Suppose that there exist $x_{\e}\in \bigcup\limits_{i=1}^k
B\left(\frac{x_{\e}^i}{\e},\frac{2\beta}{\e}\right)\setminus
B\left(\frac{x_{\e}^i}{\e},\frac{\beta}{\e}\right)$
and $R>0$, such that
\begin{equation}\label{4sec32}
\varliminf\limits_{\e\rightarrow 0}\int_{B(x_{\e},
R)}|u_{\e}|^2dx>0.
\end{equation}
Let $W_{\e}=u_{\e}(x+x_{\e})$. Using \eqref{4sec32}
, we get
\begin{equation}\label{4sec33}
\varliminf\limits_{\e\rightarrow 0}\int_{B(0,
R)}|W_{\e}|^2dx>0.
\end{equation}
Since $\e x_{\e}\in \bigcup\limits_{i=1}^k
B\left(x_{\e}^i, 2\beta\right)\setminus B\left(x_{\e}^i,
\beta\right)$, by taking a subsequence, we can assume $\e
x_{\e}\rightarrow x_0\in \bigcup\limits_{i=1}^kB\left(x^i,
2\beta\right)\setminus B\left(x^i, \beta\right)$. From
\eqref{4sec30}, one has $\{W_{\e}\}$ is bounded in
$H_{\e}$ and $H^1(\RN)$. Without loss of generality, we assume that $W_{\e}\rightharpoonup W$ weakly in $H^1(\RN)$ and strongly in $ L_{loc}^q(\RN)$ for $q\in[2,2^\ast)$. Clearly, \eqref{4sec33} implies that  $W\neq 0$ and from
\eqref{4sec27} we get that $W$ is a nontrivial solution of
\begin{equation}\label{4sec34}
-\Delta W+V(x_0)W=(I_\al\ast F(W))f(W)\,\, \mbox{in}\,\,\RN.
\end{equation}
Once choosing $R$ large enough, we deduce by the weak convergence
that
\begin{equation}\label{4sec35}\aligned
&\varliminf\limits_{\e\rightarrow 0}\int_{B(x_{\e},
R)}(|\na u_{\e}|^2+V_\e(x)u_\e^2)dx=\varliminf\limits_{\e\rightarrow
0}\int_{B(0,
R)}(|\na W_{\e}|^2+V_\e(x+x_\e)|W_\e^2|)dx\\
&\geq \int_{B(0, R)}(|\na W|^2+V(x_0)|W|^2)dx\geq
\frac{1}{2}\int_{\R^N}(|\na W|^2+V(x_0)|W|^2)dx.\endaligned
\end{equation}
By Proposition \ref{prop2.1}, $E_a$ is a mountain pass value. One can get $E_a$ is strictly increasing for $a>0$. Then
$$L_{V(x_0)}(W)\geq E_{V(x_0)}\geq E_m, \ \
\text{since} \ V(x_0)\geq m.$$ Thus by
\eqref{4sec35} and Pohozaev identity, we get
$$
\aligned &\varliminf\limits_{\e\rightarrow 0}\int_{B(x_{\e},
R)}(|\na u_{\e}|^2+V_\e(x)u_\e^2)dx\geq
\frac{1}{2}\int_{\R^N}(|\nabla W|^2+V(x_0)W^2)dx\\
&\ge\frac{N+\al}{\al+2}L_{V(x_0)}(W)\ge\frac{N+\al}{\al+2}E_{V(x_0)}(W),
\endaligned
$$
which contradicts  \eqref{4sec30} for $d>0$ small enough. It follows from \cite[Lemma I.1]{Lions1}
that
\begin{equation}\label{4sec36}
\varlimsup\limits_{\e\rightarrow
0}\int_{\bigcup\limits_{i=1}^kB\left(\frac{x_{\e}^i}{\e},\frac{2\beta}{\e}\right)\setminus
B\left(\frac{x_{\e}^i}{\e},\frac{\beta}{\e}\right)}|u_{\e}|^qdx=0,\
\ \ \mbox{for any}\,\, 2\leq q<2^\ast.
\end{equation}
As a consequence, we can derive that
\begin{equation}\label{4sec37}
\int_{\RN}(I_\al\ast F(u_\e))F(u_\e)=\int_{\RN}(I_\al\ast F(u_{1,\e}))F(u_{1,\e})+\int_{\RN}(I_\al\ast F(u_{2,\e}))F(u_{2,\e})+o_\e(1).
\end{equation}
Indeed, let $G(u_\e):=F(u_\e)-F(u_{1,\e})-F(u_{2,\e})$, then
$$
G(u_\e)(x)=0,\,\, \mbox{if}\,\,x\not\in\bigcup\limits_{i=1}^kB\left(\frac{x_{\e}^i}{\e},\frac{2\beta}{\e}\right)\setminus
B\left(\frac{x_{\e}^i}{\e},\frac{\beta}{\e}\right).
$$
By $(F1)$-$(F2)$, for any $\dd_0>0$ there exists $c>0$(depending on $\dd_0$) such that $|G(u_\e)|\le c|u_\e|^2+\dd_0|u_\e|^{(N+\al)/(N-2)}$. Then by Hardy-Littlewood-Sobolev inequality,
{\allowdisplaybreaks
\begin{align}\lab{vv}
&\left|\int_{\RN}(I_\al\ast G(u_\e))F(u_\e)\right|\nonumber\\
&\le C(N,\al)\left(\int_{\RN}|G(u_\e)|^{2N/(N+\al)}\right)^{(N+\al)/(2N)}\left(\int_{\RN}|F(u_\e)|^{2N/(N+\al)}\right)^{(N+\al)/(2N)}\nonumber\\
&\le C'(N,\al)\left[\int_{\bigcup\limits_{i=1}^kB\left(\frac{x_{\e}^i}{\e},\frac{2\beta}{\e}\right)\setminus
B\left(\frac{x_{\e}^i}{\e},\frac{\beta}{\e}\right)}(c|u_\e|^{4N/(N+\al)}+\dd_0^{2N/(N+\al)}|u_\e|^{2^\ast})\right]^{(N+\al)/(2N)}.
\end{align}}%
Recalling that $\al>N-4$, $4N/(N+\al)\in(2,2^\ast)$. By the arbitrary of $\dd_0$, it follows from \re{4sec36} and \re{vv} that
$$
\int_{\RN}(I_\al\ast G(u_\e))F(u_\e)\rg0,\,\, \mbox{as}\,\ \e\rg0.
$$
Similarly,
$$
\int_{\RN}(I_\al\ast F(u_{1,\e}))G(u_\e)\rg0,\,\int_{\RN}(I_\al\ast F(u_{2,\e}))G(u_\e)\rg0,\,\, \mbox{as}\,\ \e\rg0.
$$
Then
{\allowdisplaybreaks
\begin{align*}
&\int_{\RN}(I_\al\ast F(u_\e))F(u_\e)\\
&=\int_{\RN}[I_\al\ast(F(u_{1,\e})+F(u_{2,\e})+G(u_{\e}))](F(u_{1,\e})+F(u_{2,\e})+G(u_{\e}))\\
&=\int_{\RN}[I_\al\ast F(u_{1,\e})]F(u_{1,\e})+\int_{\RN}[I_\al\ast F(u_{2,\e})]F(u_{2,\e})\\
&\,\, \,\,\,\,+\int_{\RN}[I_\al\ast F(u_{2,\e})]F(u_{1,\e})+\int_{\RN}[I_\al\ast F(u_{1,\e})]F(u_{2,\e})+o_\e(1).
\end{align*}
}%
On the other hand,
{\allowdisplaybreaks
\begin{align*}
&\left|\int_{\RN}[I_\al\ast F(u_{1,\e})]F(u_{2,\e})\right|\\
&\le\int_{\left(\bigcup\limits_{i=1}^kB\left(\frac{x_{\e}^i}{\e},\frac{2\beta}{\e}\right)\right)\times\left(\R^N\setminus
\bigcup\limits_{i=1}^kB\left(\frac{x_{\e}^i}{\e},\frac{\beta}{\e}\right)\right)}I_\al(x-y)|F(u_{1,\e}(x))||F(u_{2,\e}(y))|\\
&=\int_{\left(\bigcup\limits_{i=1}^kB\left(\frac{x_{\e}^i}{\e},\frac{2\beta}{\e}\right)\setminus
B\left(\frac{x_{\e}^i}{\e},\frac{\beta}{\e}\right)\right)\times\left(\R^N\setminus
\bigcup\limits_{i=1}^kB\left(\frac{x_{\e}^i}{\e},\frac{\beta}{\e}\right)\right)}I_\al(x-y)|F(u_{1,\e}(x))||F(u_{2,\e}(y))|\\
&\,\,\,\,\,\,+\int_{\bigcup\limits_{i=1}^k
B\left(\frac{x_{\e}^i}{\e},\frac{\beta}{\e}\right)\times\left(\R^N\setminus
\bigcup\limits_{i=1}^kB\left(\frac{x_{\e}^i}{\e},\frac{2\beta}{\e}\right)\right)}I_\al(x-y)|F(u_{1,\e}(x))||F(u_{2,\e}(y))|\\
&\,\,\,\,\,\,+\int_{\bigcup\limits_{i=1}^k
B\left(\frac{x_{\e}^i}{\e},\frac{\beta}{\e}\right)\times
\left(\bigcup\limits_{i=1}^kB\left(\frac{x_{\e}^i}{\e},\frac{2\beta}{\e}\right)\setminus
B\left(\frac{x_{\e}^i}{\e},\frac{\beta}{\e}\right)\right)}I_\al(x-y)|F(u_{1,\e}(x))||F(u_{2,\e}(y))|\\
&:=I_1+I_2+I_3\rightarrow 0.
\end{align*}
}%
Similar as above, by Hardy-Littlewood-Sobolev inequality and \re{4sec36}, $I_1,I_3\rg 0$ as $\e\rg0$. Obviously,
$$
|x-y|\ge\frac{\beta}{\e} \,\,\mbox{if}\,\,(x,y)\in\bigcup\limits_{i=1}^k
B\left(\frac{x_{\e}^i}{\e},\frac{\beta}{\e}\right)\times\left(\R^N\setminus
\bigcup\limits_{i=1}^kB\left(\frac{x_{\e}^i}{\e},\frac{2\beta}{\e}\right)\right).
$$
Then
$$
|I_2|\le C(N,\al,\beta)\e^{N-\al}\left(\int_{\RN}c|u_\e|^2+\dd|u_\e|^{(N+\al)/(N-2)}\right)^2.
$$
Noting that $\al\in((N-4)_+,N)$, $(N+\al)/(N-2)\in(2,2^\ast)$. Then we get $I_2\rg0$ as $\e\rg0$. Similarly, $\int_{\RN}[I_\al\ast(F(u_{2,\e})]F(u_{1,\e})\rg0$ as $\e\rg0$. So we get \eqref{4sec37}.
\vskip0.1in
By \eqref{4sec37},
{\allowdisplaybreaks
$$\aligned
\Gamma_{\e}(u_{\e})&=\Gamma_{\e}(u_{1,\e})+\Gamma_{\e}(u_{2,\e})+o_\e(1)\\
&\,\,\,\,\,+\sum\limits_{i=1}^k\int_{B\left(\frac{x_{\e}^i}{\e},\frac{2\beta}{\e}\right)\setminus
B\left(\frac{x_{\e}^i}{\e},\frac{\beta}{\e}\right)}\na\left(\varphi_{\e}
\left(y-\frac{x_{\e}^i}{\e}\right)u_{\e}\right)
\na\left(\left(1-\varphi_{\e}\left(y-\frac{x_{\e}^i}{\e}\right)\right)u_{\e}\right)\\
&\,\,\,\,\,+\sum\limits_{i=1}^k\int_{B\left(\frac{x_{\e}^i}{\e},\frac{2\beta}{\e}\right)\setminus
B\left(\frac{x_{\e}^i}{\e},\frac{\beta}{\e}\right)}V_{\e}(x)\varphi_{\e}\left(y-\frac{x_{\e}^i}{\e}\right)
\left(1-\varphi_{\e}\left(y-\frac{x_{\e}^i}{\e}\right)\right)|u_{\e}|^2\\
&\geq
\Gamma_{\e}(u_{1,\e})+\Gamma_{\e}(u_{2,\e})+o_\e(1).
\endaligned$$}%
Therefore, we get \eqref{4sec31}.

{\bf Step 2.} We claim that for $d,\e>0$ small enough,
\begin{equation}\label{4sec02}\Gamma_{\e}(u_{2,\e})\geq\frac{1}{4}\|u_{2,\e}\|_{\e}^2.\end{equation}
Indeed,
\begin{equation}\label{4sec38}\aligned
\Gamma_{\e}(u_{2,\e})&\geq P_{\e}(u_{2,\e})=\frac{1}{2}\|u_{2,\e}\|_{\e}^2-\frac{1}{2}\int_{\RN}(I_\al\ast F(u_{2,\e}))F(u_{2,\e}).
\endaligned
\end{equation}
By $(F1)$-$(F2)$, for any $\rho>0$ there exists $c>0$(depending on $\rho$) such that $|F(t)|\le \rho t^2+c|t|^{(N+\al)/(N-2)}$ for $t\in\R$. Then by Hardy-Littlewood-Sobolev inequality,
{\allowdisplaybreaks
\begin{align*}
\left|\int_{\RN}(I_\al\ast F(u_{2,\e}))F(u_{2,\e})\right|&\le C(N,\al)\left(\int_{\RN}|F(u_{2,\e})|^{2N/(N+\al)}\right)^{(N+\al)/N}\\
&\le C'(N,\al)\left[\int_{\RN}(\rho^{2N/(N+\al)}|u_\e|^{4N/(N+\al)}+c|u_\e|^{2^\ast})\right]^{(N+\al)/N}\\
&\le C''(N,\al)\left(\rho^2\|u_{2,\e}\|_{4N/(N+\al)}^2+\|u_{2,\e}\|_{2^\ast}^{2(N+\al)/(N-2)}\right).
\end{align*}}%
Notice that $4N/(N+\al)\in(2,2^\ast)$ and $2(N+\al)/(N-2)>2$. By Sobolev's inequality,
$$
\left|\int_{\RN}(I_\al\ast F(u_{2,\e}))F(u_{2,\e})\right|\le C'''(N,\al)\left(\rho^2\|u_{2,\e}\|_{\e}^2+\|u_{2,\e}\|_{\e}^{2(N+\al)/(N-2)}\right)
$$
Since $\{u_{\e}\}$ is bounded, we deduce from \eqref{4sec30}
that $\|u_{2,\e}\|_{\e}\leq 4d$ for sufficiently small
$\e>0$. Thus, taking $d$ and $\rho$ small enough, we have $$\Gamma_{\e}(u_{2,\e})\geq\frac{1}{4}\|u_{2,\e}\|_{\e}^2.$$

{\bf Step 3.} For each $i=1,2,\cdots, k$, we define
\begin{eqnarray}\label{4sec43}
u_{1,\e}^i(x)=\left\{
                  \begin{array}{ll}
                    u_{1,\e}(x), \ \ &x\in O_{\e}^i,\\
                    0, \ \ &x\notin O_{\e}^i,
                  \end{array}
                \right.
\end{eqnarray}
and set
$W_{\e}^i(x)=u_{1,\e}^i\left(x+\frac{x_{\e}^i}{\e}\right)$.
Then for fixed $i\in \{1,2,\cdots,k\}$, we can assume, up to a
subsequence that as $\e\rg0$,
$$W_{\e}^i\rightharpoonup W^i\ \ \text{weakly} \in H^1(\RN),$$
and $W^i$ is a solution of
\begin{equation*}
-\Delta W^i+V(x^i)W^i=(I_\al\ast F(W^i))f(W^i),\,\,\,x\in\RN.
\end{equation*}

In the following, we prove that $W_{\e}^i\rightarrow W^i$ strongly in $H_{\e}$. First, we prove that $W_{\e}^i\rightarrow W^i$ strongly in $L^p(\RN)$ for any $p\in(2,2^\ast)$. Otherwise, there exist $x_{\e}\in\RN$ and $R>0$ such that
$$
\varliminf\limits_{\e\rightarrow 0}\int_{B(x_{\e},R)}|W_{\e}^i- W^i|^2>0.
$$
Obviously, $|x_\e|\rg\iy$ as $\e\rg0$. Let $z_\e=x_{\e}+x_\e^i/\e$, then
\be\label{4sec44}
\varliminf\limits_{\e\rightarrow 0}\int_{B(z_\e,R)}|u_{1,\e}^i|^2>0.
\ee
Since $\vp(x)=0$ for $|x|\ge2\beta$, $|x_\e|\le3\beta/\e$(In fact $|x_\e|\le2\beta/\e$). Since $\e z_{\e}\in B(x_{\e}^i, 3\beta)$, we can
assume $\e z_{\e}\rightarrow z^i\in O^i$ as
$\e\rightarrow 0$.\\

Define
$\widetilde{W}_{\e}^i(x)=u_{1,\e}^i(x+z_{\e})$, then up to a
subsequence that as $\e\rg0$,
$$\widetilde{W}_{\e}^i\rightharpoonup \widetilde{W}^i\not\equiv0\ \ \text{weakly in}\,\, H^1(\RN)$$ and $\widetilde{W}^i$ satisfies
\begin{equation*}
-\Delta \widetilde{W}^i+V(z^i)\widetilde{W}^i=(I_\al\ast F(\widetilde{W}^i))f(\widetilde{W}^i),\,\,\,x\in\RN.
\end{equation*}
Similar as in Step 1, we can get a contradiction. So $W_{\e}^i\rightarrow W^i$ strongly in $L^p(\RN)$ for any $p\in(2,2^\ast)$, which implies
\begin{equation}\label{4sec45}
\lim\limits_{\e\rightarrow
0}\int_{\RN}(I_\al\ast F(W_\e^i))F(W_\e^i)=\int_{\RN}(I_\al\ast F(W^i))F(W^i).
\end{equation}
Then given any $i=1,2,\cdots,k$, we deduce that
{\allowdisplaybreaks
\begin{equation}\label{4sec47}\aligned
&\varlimsup\limits_{\e\rightarrow0}\Gamma_{\e}(u_{1,\e}^i)\ge\varlimsup\limits_{\e\rightarrow0}P_{\e}(u_{1,\e}^i)\\
&=\varlimsup\limits_{\e\rightarrow0}\frac{1}{2}\int_{\R^N}|\na W_\e^i|^2+V_\e(x+x_\e^i/\e)|W_\e^i|^2-\frac{1}{2}\int_{\RN}(I_\al\ast F(W^i))F(W^i)\\
&\ge\frac{1}{2}\int_{\R^N}|\na W^i|^2+V(x^i)|W^i|^2-\frac{1}{2}\int_{\RN}(I_\al\ast F(W^i))F(W^i)\\
&=L_{V(x^i)}(W^i)\geq E_{m_i}.
\endaligned
\end{equation}}%
Now, by the estimate \eqref{4sec31}, we get
\begin{equation}\label{4sec48}
\aligned \varlimsup\limits_{\e\rightarrow
0}\left(\Gamma_{\e}(u_{2,\e})+\sum\limits_{i=1}^k\Gamma_{\e}(u_{1,\e}^i)\right)
\leq \varlimsup\limits_{\e\rightarrow
0}\Gamma_{\e}(u_{\e})\leq E=\sum\limits_{i=1}^kE_{m_i}.
\endaligned
\end{equation}
On the other hand, by \eqref{4sec02} and \eqref{4sec47}, by choosing $d>0$ small enough,
\begin{equation}\label{4sec49}
\varlimsup\limits_{\e\rightarrow
0}\left(\Gamma_{\e}(u_{2,\e})+\sum\limits_{i=1}^k\Gamma_{\e}(u_{1,\e}^i)\right)\geq
\sum\limits_{i=1}^kE_{m_i}.
\end{equation}
Therefore, \eqref{4sec48} and \eqref{4sec49} imply that by choosing $d>0$ small enough, for any
$i=1,2,\cdots,k$
\begin{equation}\label{4sec08}
\lim\limits_{\e\rightarrow
0}\Gamma_{\e}(u_{1,\e}^i)=E_{m_i}.
\end{equation}
By \eqref{4sec02}, $\|u_{2,\e}\|_\e\rg0$ as $\e\rg0$. By \eqref{4sec47}, we have
$L_{V(x^i)}(W^i)=E_{m_i}$. Recalling that $E_a$ is strictly increasing for $a>0$, we
obtain $x^i\in \mathcal {M}^i$ and $W^i(\cdot)=U_i(\cdot-z_i)$ for some $U_i\in S_{m_i}$ and $z_i\in \R^N$. Moreover, by \eqref{4sec47} and \re{4sec08}, we have
$$
\lim_{\e\rg0}\int_{\R^N}|\na W_\e^i|^2+V_\e(x+x_\e^i/\e)|W_\e^i|^2=\int_{\R^N}|\na W^i|^2+V(x^i)|W^i|^2.
$$
Then $W_\e^i\rg W^i$ strongly in $H^1(\RN)$. Let $y_\e^i=z_i+x_\e^i/\e$, then $\e y_\e^i\rg x^i\in \mathcal {M}^i$ and $u_{1,\e}^i\rg\vp_\e(\cdot-y_\e^i)U_i(\cdot-y_\e^i)$ in $H^1(\RN)$ as $\e\rg0$. Noting that $supp(u_{1,\e}^i)\subset O_\e^i$, $u_{1,\e}^i\rg\vp_\e(\cdot-y_\e^i)U_i(\cdot-y_\e^i)$ in $H_\e$ as $\e\rg0$, which implies that
$$u_{1,\e}=\sum\limits_{i=1}^ku_{1,\e}^i=\sum\limits_{i=1}^k\varphi_{\e}(x-y_{\e}^i)U_i(x-y_{\e}^i)+o_\e(1)\,\,\mbox{in}\,\, H_\e.$$
Since $\lim_{\e\rightarrow
0}\Gamma_{\e}(u_{2,\e})=0 $, from \eqref{4sec02}, we know $u_{2,\e}\rightarrow 0$ in $H_{\e}$, then the proof is completed.
\end{proof}
\noindent Immediately, as a consequence of Proposition \ref{prop5}, we have
\begin{proposition}\label{prop6}
For sufficiently small $d>0$, there exist constants $\eta>0$ and
$\e_0>0$, such that $|\Gamma_{\e}^{'}(u)|\geq \eta$ for
$u\in \Gamma_{\e}^{D_{\e}}\cap(X_{\e}^d\setminus
X_{\e}^{\frac{d}{2}})$ and $\e\in (0, \e_0)$.
\end{proposition}

Now, we fix $d>0$ such that Proposition \ref{prop6} holds. Choose $R_0>0$ large enough such that
$O\subset B(0, R_0)$ and $\gamma_{\e}(s)\in H_0^1(B(0,
\frac{R}{\e}))$ for any $s\in T$, $R>R_0$.

\begin{proposition}\label{prop70}
Given $\e>0$ sufficiently small, then there exists a sequence $\{u_n^R\}\subset
X_{\e}^{\frac{d}{2}}\cap \Gamma_{\e}^{D_{\e}}\cap
H_0^1(B(0, \frac{R}{\e}))$, such that
$\lim\limits_{n\rightarrow \infty}\|\Gamma^{'}(u_n^R)\|=0$ in
$H_0^1(B(0, \frac{R}{\e}))$.
\end{proposition}
\begin{proof}
The proof is similar to \cite{bj}. To the contrary, for
$\e>0$ small enough, there exists $a_R(\e)>0$ such that
$\|\Gamma_{\e}^{'}(u)\|\geq a_R(\e)$ for any $u\in
X_{\e}^{d}\cap\Gamma_{\e}^{D_{\e}}\cap H_0^1(B(0,
\frac{R}{\e}))$. It follows from Proposition \ref{prop2.3} that there exists
$\alpha_0\in (0, E-\tilde{E})$ such that if $\e>0$ small enough and $\Gamma_{\e}(\gamma_{\e}(s))\geq D_{\e}-\alpha_0$,
then $\gamma_{\e}(s)\in X_{\e}^{\frac{d}{2}}\cap H_0^1(B(0, \frac{R}{\e}))$. Thus, by a deformation argument in $H_0^1\left(B(0,
\frac{R}{\e})\right)$, there exist a $\mu_0\in (0,\alpha_0)$
and a path $\gamma\in \left(C[0, T], H_{\e}\right)$ such that
\begin{eqnarray*}
\gamma(s)\left\{
                  \begin{array}{ll}
                    =\gamma_{\e}(s) \ \ & \ \ \text{for} \ \ \gamma_{\e}(s)\in \Gamma_{\e}^{D_{\e}-\alpha_0}\\
                    \in X_{\e}^d \ \ & \ \ \text{for} \ \ \gamma_{\e}(s)\notin \Gamma_{\e}^{D_{\e}-\alpha_0},
                  \end{array}
                \right.
\end{eqnarray*}
 and
\begin{equation}\label{4sec09}
 \Gamma_{\e}(\gamma(s))< D_{\e}-\mu_0,  s\in T.\end{equation}
Let $\psi\in C_0^{\infty}(\R^N)$ be such that $\psi(x)=1$ for $x\in
 O^{\delta}$, $\psi(x)=0$ for $x\notin
 O^{2\delta}$, $\psi(x)\in [0,1]$ and $|\nabla \psi|\leq
 \frac{2}{\delta}$. For $\gamma(s)\in X_{\e}^d$, we define $\gamma_1(s)=\psi_{\e}\gamma(s)$,
 $\gamma_2(s)=(1-\psi_{\e})\gamma(s)$, where
 $\psi_{\e}=\psi(\e x)$, then
{\allowdisplaybreaks
\begin{equation*}\aligned
&\Gamma_{\e}\left(\gamma(s)\right)=\Gamma_{\e}\left(\gamma_1(s)\right)+\Gamma_{\e}\left(\gamma_2(s)\right)+Q_{\e}(\gamma(s))-Q_{\e}(\gamma_1(s))-Q_{\e}(\gamma_2(s))\\
&\,\,\,\,\,\,\,\,\,\,\,\,\,\,\,\,\,\,\,\,\,\,\,\,\,\,\,+\int_{\R^N}\left(\psi_{\e}(1-\psi_{\e})|\na\gamma(s)|^2+V_{\e}\psi_{\e}(1-\psi_{\e})|\gamma(s)|^2\right)+o_\e(1)\\
&\,\,\,\,\,\,\,\,\,\,\,\,\,\,\,\,\,\,\,\,\,\,\,\,\,\,\,-\frac{1}{2}\int_{\RN}(I_\al\ast F(\g(s)))F(\g(s))+\frac{1}{2}\int_{\RN}(I_\al\ast F(\g_1(s)))F(\g_1(s))\\
&\,\,\,\,\,\,\,\,\,\,\,\,\,\,\,\,\,\,\,\,\,\,\,\,\,\,\,+\frac{1}{2}\int_{\RN}(I_\al\ast F(\g_2(s)))F(\g_2(s)).
\endaligned\end{equation*}}%
Notice that
{\allowdisplaybreaks
\begin{equation*}\aligned
Q_{\e}(\gamma(s)) &=
\left(\int_{\R^N}\chi_{\e}|\gamma_1(s)|^2+\int_{\R^N}\chi_{\e}|\gamma_2(s)|^2-1\right)_{+}^2\\
&\geq
\left(\int_{\R^N}\chi_{\e}|\gamma_1(s)|^2-1\right)_{+}^2+\left(\int_{\R^N}\chi_{\e}|\gamma_2(s)|^2-1\right)_{+}^2\\
&=Q_{\e}\left(\gamma_1(s))+Q_{\e}(\gamma_2(s)\right).
\endaligned
\end{equation*}}%
Then we get
{\allowdisplaybreaks
\begin{align}\lab{1111}
&\Gamma_{\e}\left(\gamma(s)\right)\ge\Gamma_{\e}\left(\gamma_1(s)\right)+\Gamma_{\e}\left(\gamma_2(s)\right)+\frac{1}{2}\int_{\RN}(I_\al\ast F(\g_2(s)))F(\g_2(s))\\
&\,\,\,\,\,\,\,\,\,\,\,\,\,\,\,\,\,\,\,\,\,\,\,\,\,\,\,-\frac{1}{2}\int_{\RN}(I_\al\ast F(\g(s)))F(\g(s))+\frac{1}{2}\int_{\RN}(I_\al\ast F(\g_1(s)))F(\g_1(s))+o_\e(1)\nonumber.
\end{align}}%
Since $Q_{\e}(\gamma(s))$ is uniformly bounded with respect to
$\e$, there exists $C>0$ such that
\begin{equation}\label{4sec010}
\int_{\R^N\setminus O_{\e}}|\gamma(s)|^2\leq
C\e^{\mu},\, s\in T.
\end{equation}
Let $H(\g(s)):=F(\g(s))-F(\g_1(s))-F(\g_2(s))$, then
$$
H(\g(s))(x)=0,\,\, \mbox{if}\,\,x\not\in O_\e^{2\dd}\setminus O_\e^\dd.
$$
Then similar to \re{vv}, for any $\dd_0>0$, by virtue of the  Hardy-Littlewood-Sobolev inequality,
\begin{align*}
&\left|\int_{\RN}(I_\al\ast H(\g(s)))F(\g(s))\right|\nonumber\\
&\le C'(N,\al)\left[\int_{O_\e^{2\dd}\setminus O_\e^\dd}(c|\g(s)|^{4N/(N+\al)}+\dd_0^{2N/(N+\al)}|\g(s)|^{2^\ast})\right]^{(N+\al)/(2N)}.
\end{align*}
Noting that $\g(s)$ is uniformly bounded in $H^1(\RN)$ for $\e$ and $s\in T$, thanks to the interpolation inequality and \re{4sec010}, we have
$$
\lim_{\e\rg0}\int_{\RN}[I_\al\ast H(\g(s))]F(\g(s))=0.
$$
Similarly, $i\not=j$ and $i,j=1,2$,
$$
\lim_{\e\rg0}\int_{\RN}[I_\al\ast F(\g_i(s))]H(\g(s))=0,\,\,\lim_{\e\rg0}\int_{\RN}[I_\al\ast F(\g_i(s))]F(\g_j(s))=0,
$$
and
$$
\lim_{\e\rg0}\int_{\RN}[I_\al\ast F(\g_2(s))]F(\g_2(s))=0.
$$
Then
\begin{align*}
\int_{\RN}[I_\al\ast F(\g(s))]F(\g(s))=\int_{\RN}[I_\al\ast F(\g_1(s))]F(\g_1(s))+o_\e(1).
\end{align*}
By \re{1111}, \begin{equation}\label{2}
\Gamma_{\e}(\gamma(s))\geq
\Gamma_{\e}(\gamma_1(s))+o_\e(1).
\end{equation}

For $i=1,2,\cdots, k$, let
\begin{eqnarray*}
\gamma_1^i(s)(x)=\left\{
                  \begin{array}{ll}
                    \gamma_1(s)(x), \ &\text{for} \ x\in (O^i)_{\e}^{2\delta}, \\
                    0, \ &\text{for} \ x\notin (O^i)_{\e}^{2\delta},
                  \end{array}
                \right.
\end{eqnarray*}
then
\begin{equation}\label{4sec011}
\Gamma_{\e}(\gamma_1(s))\geq
\sum\limits_{i=1}^k\Gamma_{\e}(\gamma_1^i(s))=\sum\limits_{i=1}^k\Gamma_{\e}^i(\gamma_1^i(s)).
\end{equation}
Since $0<\alpha_0<E-\tilde{E}$, by Proposition
\ref{prop2.3}, for all $i\in \{1,2,\cdots,k\}$,
$\gamma_1^i(s)\in \Phi_{\e}^i$. Thus, thanks to
\cite[Proposition 3.4]{CR} and \eqref{4sec011}, we
deduce that
$$\max\limits_{s\in T}\Gamma_{\e}(\gamma(s))\geq E+o_\e(1).
$$
Combining with \eqref{4sec09}, we get $E\leq D_{\e}-\mu_0$, which is a contradiction.
\ep
\begin{proposition}\label{prop7}
Given $\e,d>0$ sufficiently small, $\Gamma_{\e}$ has a nontrivial critical point $u\in X_{\e}^d\cap \Gamma_{\e}^{D_{\e}}$.
\end{proposition}
\bp
Let $\left\{u_n^R\right\}$ be a Palais-Smale sequence of $\Gamma_{\e}$ obtained above, then due to $u_n^R\in X_\e^d\cap\Gamma_{\e}^{D_{\e}}$, $\left\{u_n^R\right\}$ is uniformly bounded in $H_0^1\left(B(0,\frac{R}{\e})\right)$ for $n$. Up to a subsequence, as $n\rg\iy$, $u_n^R\rightarrow u_\e^R$ strongly in
$H^1_0\left(B(0, \frac{R}{\e})\right)$ and $u_\e^R$ is a critical
point of $\Gamma_{\e}$ on $H^1_0\left(B(0,
\frac{R}{\e})\right)$ and satisfies
\begin{equation}\label{4sec001}
-\DD u_\e^R+\ti{V}_{\e}u_\e^R=[I_\al\ast F(u_\e^R)]f(u_\e^R),\,\,|x|\le R/\e,
\end{equation}
where $$\ti{V}_\e=V_\e+4\left(\int_{\R^N}\chi_{\e}|u^R|^2\,\ud x-1\right)_+\chi_{\e}.$$ Since $f(t)=0$ for $t<0$, one knows $u_\e^R\ge0$ in $B\left(0, R/\e\right)$. We extend $u_\e^R\in H_0^1(B(0,R/\e))$ to $u_\e^R\in H_\e$ by zero outside $B(0,R/\e)$. Noting that $u_\e^R\in X_\e^R$, $\{u_\e^R\}$ is uniformly bounded in $H_\e$ for $R,\e$. By repeating the argument in \cite[Proposition 3.1]{MV1}, for any $p\in[2,\frac{N}{\al}\frac{2N}{N-2})$, there exists $C_p$ (independent of $\e,R$) such that$\|u_\e^R\|_p\le C_p\|u_\e^R\|_2.$ Then similar as in Proposition \ref{decay}, we know $I_\al\ast F(u_\e^R)$ is uniformly bounded in $L^\iy(\RN)$ for $\e,R$. So there exists $C$(independent of $\e,R$) such that
$$
-\DD u_\e^R+u_\e^R\le Cf(u_\e^R),\,\,|x|\le R/\e.
$$
Thanks to $(N+\al)/(N-2)<2^\ast$, it follows from the standard the Moser iteration \cite{GT} that $\{u_\e^R\}$ is uniformly bounded in $L^\iy(\RN)$ for $\e,R$. On the other hand, by $u_n^R\in X_\e^d\cap\G_\e^{D_\e}$, there exists $C>0$(independent of $\e,n,R$) such that $$\int_{\R^N}\chi_{\e}|u_n^R|^2\,\ud x\le C,\,\,n\in\mathbb{N}.$$ By Fatou' Lemma, $\int_{\R^N\setminus O_\e}|u_\e^R|^2\,\ud x\le C\e^\mu$ for all $\e,R$. Then it follows from \cite[Theorem 8.17]{GT} and the comparison principle that there exist $c,C>0$(independent of $\e,R$) such that
\be\lab{decay11}
u_\e^R(x)\le C\exp{(-c|x|)},\,\,x\in\RN,
\ee
which yields that, up to a subsequence, $u_\e^R\rg u_\e$ strongly in $L^p(\RN)$ as $R\rg\infty$ for any $p\in[2,2^\ast)$. Thus, $u_\e\in X_\e^d\cap\G_\e^{D_\e}$ is a nontrivial critical point of $\G_\e$. Obviously, $0\not\in X_\e^d$ if $d>0$ small enough. So $u_\e\not\equiv0$ if $d>0$ small.
\end{proof}

\subsection{Completion of the proof for Theorem \ref{Theorem 1}}
\bp
\noindent By Proposition \ref{prop7}, there exist $d>0$ and $\e_0>0$, such that $\G_\e$ has a nontrivial critical point $u_\e\in X_\e^d\cap\G_\e^{D_\e}$ for $\e\in(0,\e_0)$. Since $u_\e^R\rg u_\e$ strongly in $L^2(\RN)$ as $R\rg\infty$ and \re{decay11}, there exists $C>0$ such that $\sup_{\e\in(0,\e_0)}\|u_\e\|_\iy\le C$. By $(F1)$ and $u_\e\not\equiv0$,
\be\lab{rr1}
\inf_{\e\in(0,\e_0)}\|u_\e\|_\iy\ge\rho.
\ee
Since $f(t)=0$ for $t\le 0$, we see that $u_\e\ge 0$. By \re{rr1} an the weak Harnark
inequality (see \cite{GT}), $u_\e>0$ in $\RN$. By Proposition \ref{prop5}, there exist
$\{y_\e^i\}_{i=1}^k\subset\RN, x^i\in \mathcal{M}^i, U_i\in S_{m_i}$ such that for any $1\le i\le k$,
$$
\lim_{\e\rg 0}|\e y_\e^i-x^i|=0\ \mbox{and}\ \lim_{\e\rg 0}\|u_\e-\sum_{i=1}^kU_i(\cdot-y_\e^i)\|_\e=0.
$$
Let $w_\e^i(y)=u_\e(y+y_\e^i)$, then $\lim_{\e\rg 0}\|w_\e^i-U_i\|_2=0$, which implies that for any $\sigma>0$, there exists $R>0$ (independent of $\e,i$)
such that
$$
\sup_{\e\in(0,\e_0)}\int_{\RN\setminus B(0,R)}(w_\e^i)^2\le\sigma.
$$
Similar as above, $I_\al\ast F(u_\e)$ is uniformly bounded in $L^\iy(\RN)$ for $\e$. Then it follows from \cite[Theorem 8.17]{GT} and the comparison principle that for each $1\le i\le k$, there exist $M>0$ (independent of $\e,i$) and
$y_\e^i\in\RN$, such that
$$0<w_\e^i(y)\le
M\exp\left(-\frac{|y|}{2}\right)\ \mbox{for}\ y\in\RN,
\e\in(0,\e_0).
$$
Then
\be\lab{rr2}
0<u_\e(y)\le M\exp\left(-\frac{1}{2}\min_{1\le i\le k}|y-y_\e^i|\right)\ \mbox{for}\ y\in\RN,
\e\in(0,\e_0),
\ee
which yields that $Q_\e(u_\e)=0$ for small $\e>0$. Therefore, $u_\e$ is a critical point of $P_\e$. This completes the proof.
\ep


\begin{thebibliography}{100}

\bibitem{Alves1} C. O. Alves, J. Marcos do \'{O} and M. A. S. Souto, Local mountain-pass for a class of elliptic problems in $\mathbb{R}^N$ involving critical growth, {\it Nonlinear Anal., }{\bf  46}(2001),  495-510.

\bibitem{alves} C.O. Alves, M. Yang, Investigating the multiplicity and concentration behaviour of solutions for a
quasi-linear Choquard equation via the penalization method. {\it Proc. Roy. Soc. Edinburgh}, {\bf 146 A}(2016), 23-58.

\bibitem{alves1} C.O. Alves, M. Yang, Existence of semiclassical ground state solutions for a generalized Choquard equation. {\it J. Differential Equations}, {\bf 257}(2014), 4133--4164.

\bibitem{Lions} H. Berestycki and P. L. Lions, Nonlinear scalar field equations I. Existence of a ground state, {\it Arch. Ration. Mech. Anal.,} {\bf 82}(1983), 313-346.

\bibitem{byeon} J. Byeon, L. Jeanjean, Standing waves for nonlinear Sch\"{o}dinger equations with a general nonlinearity. {\it Arch. Ration. Mech. Analysis}, {\bf 185}(2007), 185-200.

\bibitem{bj}J. Byeon, L. Jeanjean, Multi-peak standing waves for nonlinear Schr\"{o}dinger equations with a general nonlinearity. {\it Discrete
Contin. Dynam. Syst.}, {\bf 19}(2007), 255-269.

\bibitem{BT1} J. Byeon and K. Tanaka, Semi-classical standing waves for nonlinear Schr\"{o}dinger equations at structurally stable critical points of the potential, {\it J. Eur. Math. Soc., } {\bf 15} (2013), 1859-1899.

\bibitem{byeon4} J. Byeon, singularly perturbed nonlinear Dirichlet problems with a general nonlinearity, {\it Trans. Amer. Math. Soc. , } {\bf 362}(2010),  1981-2001.

\bibitem{BT2} J. Byeon and K. Tanaka, Semiclassical standing waves with clustering peaks for nonlinear
Schr\"{o}dinger equations, {\it Memoirs of the American Mathematical
Society,} {\bf 229}(2014).

\bibitem{BW} J. Byeon, Z.-Q. Wang, Standing waves with critical frequency for nonlinear schrodinger equations II. {\it Calc. Var. Partial Differ. Equ.}, {\bf 18}(2003), 207-219.

\bibitem{Chen} W. X. Chen, C. M. Li, B. Ou, Classification of solutions for an integral equation, {\it Comm. Pure. Appl. Math.}, {\bf 59}(2006), 330--343.


\bibitem{CSS} S. Cingolani, S. Secchi, M. Squassina, Semiclassical limit for Schr\"{o}dinger equations with magnetic field and Hartree-type nonlinearities, {\it Proc. Roy. Soc. Edinburgh}, {\bf 140 A}(2010), 973-1009.

\bibitem{DF} M. del Pino, P. Felmer, Local mountain passes for semilinear elliptic problems in unbounded domains, {\it Calc. Var. Partial
Differ. Equ.}, {\bf 4} (1996), 121-137.

\bibitem{Felmer2} M. del Pino and P. Felmer, Multi-peak bound states of nonlinear Schr$\ddot{o}$dinger
equations, {\it Ann. Inst. H. Poincar$\acute{e}$ Anal. Non-lin$\acute{e}$aire},  {\bf  15}(1998), 127-149.

\bibitem{Pino} M. del Pino and P. L. Felmer, Spike-layered solutions of singularlyly perturbed elliptic problems in a
degenerate setting, {\it Indiana Univ. Math. J., } {\bf 48} (1999) 883-898.

\bibitem{DPR} P. D'Avenia, A. Pomponio and D. Ruiz, Semi-classical states for the Nonlinear
Schr\"{o}dinger Equation on saddle points of the potential via variational methods, {\it J.
Funct. Anal.,} {\bf 262} (2012), 4600-4633.


\bibitem{F-W} A. Floer and A. Weinstein, Nonspreading wave packets for the cubic Schr\"{o}dinger
equations with a bounded potential, {\it J. Funct. Anal., } {\bf 69}(1986), 397-408.

\bibitem{GT} D. Gilbarg, N. S. Trudinger, Elliptic partial differential equations  of second order. second edition, Grundlehren 224, Springer, Berlin, Heidelberg, New York and Tokyo, 1983.

\bibitem{Gross} E. P. Gross, Physics of many-Particle systems. Vol.1, Gordon Breach, New York, 1996.

\bibitem{Lieb1} E. H. Lieb, Existence and uniqueness of the minimizing solution of Choquard's nonlinear equation. {\it Stud. Appl. Math.}, {\bf57}(1977), 93-105 .

\bibitem{LL} E. H. Lieb, M. Loss, Analysis, 2nd edn. Graduate Studies in Mathematics, vol. 14. American Mathematical Society, Providence, 2001.

\bibitem{LS} E. H. Lieb, B. Simon, The Hartree-Fock theory for Coulomb systems. {\it Comm. Math. Phys.}, {\bf 53}(1977), 185-194.

\bibitem{Lions2} P. L. Lions, The Choquard equation and related questions. {\it Nonlinear Anal. TMA}, {\bf4}(1980), 1063-1073.

\bibitem{Lions3} P. L. Lions, Compactness and topological methods for some nonlinear variational problems of mathematical physics.
Nonlinear problems: present and future 17-34, 1982.

\bibitem{Lions1} P. L. Lions, The concentration-compactness principle in the calculus of variations. The locally compact case I. II. {\it Annales Inst.
H. Poincar\'{e} Analyse Non Lin\'{e}aire}, {\bf 1}(1984), 109-145, 223-283.

\bibitem{MZ} L. Ma, L. Zhao, Classification of positive solitary solutions of the nonlinear Choquard equation. {\it Arch. Rational Mech. Anal.}, {\bf195} (2010), 455-467.

\bibitem{MN} M. Macr\`{\i}, M. Nolasco, Stationary solutions for the non-linear Hartree equation with a slowly varying potential. {\it NoDEA}, {\bf 16}(2009), 681-715.


\bibitem{MV1} V. Moroz, J. Van Schaftingen, Existence of ground states for a class of nonlinear Choquard equations. {\it Trans. Amer. Math. Soc., }{\bf 367}(2015), 6557--6579.

\bibitem{MV2} V. Moroz, J. Van Schaftingen, Semi-classical states for the Choquard equation.  {\it Calc. Var. Partial Differ. Equ.}, {\bf52}(2015), 199-235.

\bibitem{MV3} V. Moroz, J. Van Schaftingen, Groundstates of nonlinear Choquard equations: existence, qualitative properties and decay asymptotics, {\it J. Funct. Anal.}, {\bf 265}(2013), 153--184.


\bibitem{Ni-Wei} W. M. Ni and J. Wei, On the location and profile of spike-layer solutions to singularly perturbed semilinear Dirichlet problems, {\it Commun. Pure Appl. Math., }  {\bf 48} (1995) 731-768.

\bibitem{Nolasco} M. Nolasco, Breathing modes for the Schr\"{o}dinger-Poisson system with a multiple-well external potential. {\it Comm. Pure Appl. Anal.}, {\bf9}(2010), 1411-1419.

\bibitem{Oh} Y. G. Oh,  Existence of semiclassical bound states of nonlinear Schr\"{o}dinger equations
with potentials of the class $(V)_a$, {\it Comm. Partial Differential
Equations,} {\bf 13}(1988), 1499-1519.

\bibitem{Penrose1} R. Penrose, On gravity's role in quantum state reduction, {\it Gen. Rel. Grav.}, {\bf 28}(1996), 581--600.

\bibitem{Penrose2} R. Penrose, Quantum computation, entanglement and state reduction, {\it R. Soc. Lond. Philos. Trans. Ser. A Math. Phys. Eng. Sci.}, {\bf356} (1998), 1927--1939.

\bibitem{Penrose3} R. Penrose, The road to reality. A complete guide to the laws of the universe, Alfred A. Knopf Inc., New York 2005.

\bibitem{Rab} P. H. Rabinowitz, On a class of nonlinear Schr\"{o}dinger equations, {\it Z. Angew. Math. Phys., } {\bf 43}(1992), 270-291.

\bibitem{Secchi} S. Secchi, A note on Schr\"{o}dinger-Newton systems with decaying electric potential, {\it Nonlinear Anal.}, {\bf72}(2010), 3842-3856.

\bibitem{Sun} X. Sun and Y. Zhang, Multi-peak solution for nonlinear magnetic Choquard type equation, {\it J. Math. Phys.}, {\bf55}(2014), 031508.

\bibitem{WW} J. Wei, M. Winter, Strongly interacting bumps for the Schr\"{o}dinger-Newton equation, {\it J. Math. Phys.}, {\bf50}(2009), 012905.

\bibitem{WX} X. Wang, On concentration of positive bound states of nonlinear Schr\"{o}dinger equations, {\it Comm. Math. Phys., } {\bf 153}(1993), 229-244.

\bibitem{YD} M. Yang , Y. Ding, Existence of solutions for singularly perturbed Schr\"odinger equations with nonlocal part, {\it Comm. Pure Appl. Anal.},
{\bf12}(2013), 771--783.

\bibitem{CR} V. C. Zelati, P. H. Rabinowitz, Homoclinic orbits for second order Hamiltonian systems possessing superquadratic potentials. {\it J. Amer. Math. Soc.}, {\bf4}(1991), 693--727.


\end{thebibliography}
\end{document}